\documentclass{article}
\usepackage{amssymb}
\usepackage{amsmath}
\usepackage{amsthm}
\usepackage{mathtools}
\usepackage{algpseudocode}
\usepackage{enumerate}
\usepackage{tikz}

\newtheorem{theorem}{Theorem}[section]
\newtheorem{lemma}{Lemma}[section]
\newtheorem{claim}{Claim}[section]

\newcommand{\Chi}{\protect\raisebox{2pt}{$\chi$}}

\DeclarePairedDelimiter{\ceil}{\lceil}{\rceil}

\begin{document}

\thispagestyle{empty}

\title{On the Chromatic Number of $\mathbb{R}^n$ for Small Values of $n$}

\author{
Geoffrey Exoo \\
Department of Mathematics and Computer Science \\
Indiana State University \\
Terre Haute, IN 47809 \\
ge@cs.indstate.edu \\
\and
Dan Ismailescu \\
Mathematics Department \\
Hofstra University \\
Hempstead, NY 11549  \\
dan.p.ismailescu@hofstra.edu \\
}
\maketitle

\begin{abstract}
The lower bound for the chromatic number of $\mathbb{R}^n$ is
improved for $n = 6, 7, 10, 11, 12, 13 \mbox{ and } 14$.
\end{abstract}

\section{\bf Introduction.}

The chromatic number of $n$-dimensional Euclidean space,
denoted $\Chi(\mathbb{R}^n)$, is the minimum number of colors that can
be assigned to the points of $\mathbb{R}^n$ so that
no two points at distance one receive the same color.
In this note, we establish new lower bounds for $\Chi(\mathbb{R}^n)$ for several
small values of $n$.

In \cite{kr9}, a table of lower bounds for the $\Chi(\mathbb{R}^n)$ was given.
Besides the new bounds given in that paper, we are aware of only one
other improvement \cite{eik}.  Based on this improvements, we give a modified
table below.  The table also indicates the new bounds given here.

\begin{table}[!ht]
\centering
\begin{tabular}{|r|rl|r|} \hline
n \hspace{1mm} & \multicolumn{2}{|c|}{known bound} & new bound \\ \hline
 2 \hspace{1mm} & \hspace{2mm}   4 & \cite{moser} & \\
 3 \hspace{1mm} & \hspace{2mm}   6 & \cite{nechushtan} & \\
 4 \hspace{1mm} & \hspace{2mm}   9 & \cite{eik} & \\
 5 \hspace{1mm} & \hspace{2mm}   9 & \cite{cantwell} & \\
 6 \hspace{1mm} & \hspace{2mm}  11 & \cite{cibulka} & 12 \\
 7 \hspace{1mm} & \hspace{2mm}  15 & \cite{borsukprob} & 16 \\
 8 \hspace{1mm} & \hspace{2mm}  16 & \cite{szekely} & \\
 9 \hspace{1mm} & \hspace{2mm}  21 & \cite{kr9} & \\
10 \hspace{1mm} & \hspace{2mm}  23 & \cite{kr9} & 26 \\
11 \hspace{1mm} & \hspace{2mm}  24 & \cite{kr9} & 26 \\
12 \hspace{1mm} & \hspace{2mm}  24 & \cite{larmanrogers} & 36 \\
13 \hspace{1mm} & \hspace{2mm}  31 & \cite{szekely} & 36 \\
14 \hspace{1mm} & \hspace{2mm}  35 & \cite{szekely} & 36 \\
15 \hspace{1mm} & \hspace{2mm}  37 & \cite{szekely} & \\ \hline
\end{tabular}
\caption{Lower bounds on $\Chi(\mathbb{R}^n)$ for small $n$.}
\end{table}

\section{\bf $\chi(\mathbb{R}^6)$}

We construct a graph $G_{175}$ of order $175$ with chromatic number $12$.
The vertices of the graph are a set of points in $\mathbb{R}^6$ 
generated by $11$ special points.  The coordinates of each of these
special points is permuted in all $6! = 720$ possible ways to obtain
the full set of $175$ vertices.  The graph will be constructed as an
$r$-distance graph for $r = \sqrt{8}$.  The coordinates
for each point can be divided by $r$ to obtain a unit distance
graph.

The following table lists the aforementioned $11$ special points along with
the number of distinct points generated by permuting their coordinates.
For each of these $11$ points, $v_i$, let $V_i$ denote the set of points
obtained by permuting the coordinates, and let $n_i = | V_i |$, as shown
in the table.

\vspace{5mm}

\begin{table}[!ht]
\centering
\begin{tabular}{|r|rrrrrr|r|} \hline
Point & \multicolumn{6}{|c|}{Coordinates} & $n_i$ \\ \hline
$v_{1}$ & $0$  & $0$  & $0$  & $0$  & $0$  & $0$ & 1 \\
$v_{2}$ & $2$  & $2$   & $0$  & $0$  & $0$  & $0$ & 15 \\
$v_{3}$ & $2$  & $2$   & $2$  & $2$  & $0$  & $0$ & 15 \\
$v_{4}$ & $\sqrt{3}$ & $1$   & $1$  & $1$  & $1$  & $1$ & 6 \\
$v_{5}$ & $\sqrt{3}$ & $1$   & $1$  & $1$  & $1$  & $-1$ & 30 \\
$v_{6}$ & $-\sqrt{3}$ & $1$   & $1$  & $1$  & $1$  & $1$ & 6 \\
$v_{7}$ & $-\sqrt{3}$ & $1$   & $1$  & $1$  & $1$  & $-1$ & 30 \\
$v_{8}$ & $2+\sqrt{3}$ & $1$   & $1$  & $1$  & $1$  & $1$ & 6 \\
$v_{9}$ & $2+\sqrt{3}$ & $1$   & $1$  & $1$  & $1$  & $-1$ & 30 \\
$v_{10}$ & $2-\sqrt{3}$ & $1$   & $1$  & $1$  & $1$  & $1$ & 6 \\
$v_{11}$ & $2-\sqrt{3}$ & $1$   & $1$  & $1$  & $1$  & $-1$ & 30 \\ \hline
 & \hspace{1cm} & \hspace{1cm} & \hspace{1cm} & \hspace{1cm} & \hspace{1cm} & \hspace{1cm} & 175 \\ \hline
\end{tabular}
\caption{The $11$ points that generate the vertices of $G_{175}$.}
\end{table}

\vspace{5mm}

Observe that the subgraphs induced by $V_2$ and by $V_3$ are each isomorphic
to the line graph of $K_6$.  In the case of $V_2$ this is because two points
in $V_2$ are adjacent if their dot product is $4$.
This occurs when there is exactly one coordinate position where both points
have a $2$.
In the case of $V_3$, two points are adjacent if there is exactly one coordinate
position where both points have a zero.
Next define
an isomorphism $\phi: V_2 \rightarrow V_3$ by letting
\[
\phi(u_1,u_2,u_3,u_4,u_5,u_6) = (2 - u_1, 2 - u_2, 2 - u_3, 2 - u_4, 2 - u_5, 2 - u_6).
\]

Then the edges joining $V_2$ and $V_3$ are given as follows.  A vertex $x \in V_2$ is
adjacent to a vertex $y \in V_3$ whenever $y$ is not adjacent to (or equal to) $\phi(x)$.
This gives the subgraph $H$ of the $\sqrt{8}$-distance graph induced by $V_2 \cup V_3$.

Note that the independence number $\alpha(L(K_6)) = 3$.
However, in $H$ an independent set of size three consisting of vertices from $V_2$ dominates
every vertex in $V_3$.
So $\alpha(H) = 4$.
In fact, it can seen that every maximum independent set can
be obtained from the following independent set by an appropriate permutation of coordinates

\begin{center}
\begin{tabular}{rrrrrr}
(2, & 2, & 0, & 0, & 0, & 0) \\
(0, & 0, & 2, & 2, & 0, & 0) \\
(2, & 0, & 2, & 0, & 2, & 2) \\
(2, & 0, & 0, & 2, & 2, & 2) \\
\end{tabular}
\end{center}

\noindent
Thus we have the following lemma.

\begin{lemma}
The graph $H$ has chromatic number $8$.
\end{lemma}

The proof of the following theorem can be completed by a short
computer search that makes strong use of this lemma.

\begin{theorem}
The graph $G$ has chromatic number $12$ and can be
represented as a unit distance graph in $\mathbb{R}^6$.
\end{theorem}

\section{\bf $\chi(\mathbb{R}^7)$}

\begin{theorem}
$\chi(\mathbb{Q}^7)\ge 16$.
\end{theorem}
\begin{proof}

Consider the following fourteen sets in $\mathbb{Q}^7$:
\begin{align*}
S_{123}&=[\pm 2, \pm 2, \pm 2, 0,0,0,0],\quad T_{123}=[0,0,0,\pm 1, \pm 1, \pm 1, \pm 1],\\
S_{145}&=[\pm 2, 0,0,\pm 2, \pm 2, 0,0],\quad T_{145}= [0,\pm 1, \pm 1, 0,0, \pm 1, \pm 1],\\
S_{167}&=[\pm 2, 0,0,0,0,\pm 2, \pm 2],\quad T_{167}=[0,\pm 1, \pm 1, \pm 1, \pm 1, 0, 0],\\
S_{247}&=[0, \pm 2, 0, \pm 2,0,0, \pm 2],\quad T_{247}=[\pm 1,0, \pm 1,0, \pm 1, \pm 1, 0],\\ 
S_{256}&=[0,\pm 2, 0,0,\pm 2, \pm 2, 0], \quad T_{256}=[\pm 1,0, \pm 1, \pm 1,0, 0, \pm 1],\\
S_{346}&=[0,0,\pm 2,\pm 2,0,\pm 2,0],\quad T_{346}=[\pm 1, \pm 1,0,0, \pm 1, 0,\pm 1],\\
S_{357}&=[0, 0, \pm 2, 0, \pm 2,0, \pm 2],\quad T_{357}=[\pm 1, \pm 1,0, \pm 1,0, \pm 1, 0].\\.
\end{align*}
Denote
\begin{align*}
S&=S_{123} \cup S_{145} \cup S_{167} \cup S_{247} \cup S_{256} \cup S_{346} \cup S_{357},\\
T&=T_{123} \cup T_{145} \cup T_{167} \cup T_{247} \cup T_{256} \cup T_{346} \cup T_{357}.
\end{align*}

Let $G$ be the graph whose vertices are the points in $S\cup T$. Two vertices
are adjacent if and only if their distance is $4$.
It can be checked
that $G$ has 168 vertices and 4396 edges.
We are going to prove that $\chi(G)=16$.

Let $H$ be the subgraph of $G$ induced by the points in $S$. One can verify that
$H$ is a graph of order $|V(H)|=56$, 
size $|E(H)|= 756$, independence number $\alpha(H)=4$ and chromatic number
$\chi(H)=|V(H)|/\alpha(H)=14$.

Similarly, let $K$ be the subgraph of $G$ induced by the points in $T$.
One can verify that $K$ is a matching of order 
$|V(K)|=112$, size $|E(K)|= 56$, independence number $\alpha(K)=56$ and
chromatic number $\chi(K)=|V(K)|/\alpha(K)=2$.
It follows that $\chi(G)\le \chi(H)+\chi(K) = 14+2 = 16$.

Let $M$ be an independent set of $G$. From the observation above
$|M\cap V(H)|\le 4$. 
We say that $M$ is an {\it{independent set of type $k$}}
if $|M\cap V(H)|=k$ for some $0\le k\le 4$.
The following claim can be easily checked

\begin{claim}\label{C}
Let $M$ be an independent set of type $k$ in $G$. Then, the following hold:

\begin{align*}
\text{If}\quad &k=0,\quad \text{then} \quad |M\cap V(K)|\le 56.\\
\text{If}\quad &k=1,\quad \text{then} \quad |M\cap V(K)|\le 24.\\
\text{If}\quad &k=2,\quad \text{then} \quad |M\cap V(K)|\le 24.\\
\text{If}\quad &k=3,\quad \text{then} \quad |M\cap V(K)|\le 3.\\
\text{If}\quad &k=4,\quad \text{then} \quad |M\cap V(K)|\le 3.\\
\end{align*}
\end{claim}

Suppose that $\chi(G)\le 15$. Then the set of vertices of $G$ can be
partitioned into 15 independent sets.
Denote by $m_k$ the number of independent sets of type $k$ in this
partition, $0\le k\le 4$.
Then from Claim \ref{C}, the following relations hold true:

\begin{align*}
15&=m_0+m_1+m_2+m_3+m_4.\\
56&=m_1+2m_2+3m_3+4m_4.\\
112&\le 56m_0+24m_1+24m_2+3m_3+3m_4.
\end{align*}

But it is easy to check (either by hand or a short program) that this system
has no solutions in nonnegative integers.
Thus $\chi(G)\ge 16$.
\end{proof}

\section{\bf $\chi(\mathbb{R}^{10})$}

The construction for $\chi(\mathbb{R}^{10})$ is related to the well-known
Frankl-Wilson construction \cite{franklwilson} which established
an exponential lower bound for $\chi(\mathbb{R}^{10})$, and also gives
the best constructive lower bound for classical diagonal Ramsey numbers.

The vertices of this graph are identified with the 
points $(x_1 , x_2 , \dots , x_{11})$ in $\mathbb{R}^{11}$ such
that each $x_i = 0 \mbox{ or } 1$, and
\[
\sum\limits_{i=1}^{11} x_i = 5
\]
There are $\binom{11}{5} = 462$ such points.
Two points are adjacent if
their distance is $2$ (and so their Hamming distance is $4$).
This graph is regular of degree $\binom{5}{3} \binom{6}{2} = 150$.

A computation reveals that the independence number
of $G$, denoted as usual by $\alpha(G)$, is $18$.
We did this computation two ways.  First, our own special program
written for graph of this type was used.  Second, the result was
verified by the {\tt mcqd} program of Konc and Jane\v{z}i\v{c}
\cite{match}.
Using the fact that $\chi(G) \geq \frac{n}{\alpha{G}}$, for any
graph $G$ of order $n$, we find that
$\chi(G) \geq \ceil{\frac{462}{18}} = 26$.

Finally, for each point, the sum of the coordinates is $5$,
so the points are located on a $10$-dimensional hyperplane, so
we have the following.

\begin{theorem}
$\chi(\mathbb{R}^{n}) \geq 26$, for $n=10,11$.
\end{theorem}

\section{\bf $\chi(\mathbb{R}^{12})$}

The construction for $\mathbb{R}^{12}$ parallels that of $\mathbb{R}^{10}$.
In this case the vertex set consists of all $0-1$-vectors in $\mathbb{R}^{13}$ with Hamming
weight $6$.  Again two vertices are adjacent if their distance is $2$ (Hamming distance $4$).
So the the graph has ordeer $\binom{13}{6} = 1716$  and degree
$\binom{13}{6} \binom{13}{6}$

For this case, computing the independence number is a much harder computation.
But again, both {\tt mcqd} and our specialized program were able to determine that
the independence number is $46$.
Hence the chromatic number is at least $\ceil{\frac{1716}{46}} = 36$

\begin{theorem}
$\chi(\mathbb{R}^{n}) \geq 36$, for $n=12,13,14$.
\end{theorem}

\end{document}